\documentclass{article}
\usepackage{amsmath,amssymb,amsfonts}
\usepackage{amsthm}
\usepackage{latexsym}
\usepackage{bbold} 
\newcommand{\R}{{\mathbb R}}

\newcommand{\N}{{\mathbb N}}

\newcommand{\non}{{\nonumber}}

\renewcommand{\d}{{\rm d}}

\newtheorem{theorem}{Theorem}
\newtheorem{lemma}{Lemma}

\usepackage[utf8]{inputenc}
\usepackage{authblk}
\usepackage{esint}

\title{Conical singularities in thin elastic sheets}
\author[1]{Stefan M\"uller\thanks{stefan.mueller@hcm.uni-bonn.de}}
\author[1]{Heiner Olbermann\thanks{heiner.olbermann@hcm.uni-bonn.de}}
\affil[1]{Hausdorff Center for Mathematics \& Institute for Applied Mathematics, 
University of Bonn, Germany}
\date{\today}
\begin{document}
\maketitle
\begin{abstract}
\noindent
When one  slightly pushes  a thin elastic sheet at its center into a hollow
cylinder, the sheet forms  (to a high degree of approximation)  a 
developable cone, or 
 \emph{d-cone} for short. Here we
investigate one particular aspect of  d-cones, namely  the scaling of the  elastic energy with the sheet
thickness $h$. Following recent work of Brandman, Kohn and Nguyen  \cite{Brandman} we study  the  Dirichlet problem of finding the configuration of minimal
elastic energy when the boundary values are given by an exact \emph{d-cone}.
We improve their result 
for the energy scaling. In particular, we 
show that the deviation from the logarithmic energy scaling is bounded by 
a constant times the double logarithm of the thickness.

\end{abstract}
\section{Introduction}
Stress and energy focusing in thin elastic sheets 
has recently attracted a lot of interest in the physics literature
\cite{PhysRevE.71.016612,PhysRevLett.80.2358,Cerda08032005,RevModPhys.79.643,CCMM,PhysRevLett.87.206105,PhysRevLett.78.1303,Lobkovsky01121995,PhysRevLett.90.074302}.
One basic feature are (almost) conical singularities. A conical singularity 
arises, e.g.,  in the following experiment. Put an elastic sheet of
radius $1$ concentrically on top of a hollow cylinder of radius $R < 1$ and push the sheet
down at its centre. It has been observed that the sheet assumes (to a high degree of approximation) the
shape of a developable cone (or \emph{d-cone} for short). 
In the physics literature, this has been
discussed e.g.~in
\cite{PhysRevLett.80.2358,Cerda08032005,PhysRevE.71.016612,RevModPhys.79.643}. There
are several remarkable features of the d-cone: The angle subtended by the
region where the sheet lifts off the rim of the container is a universal
constant (approx.~139$^\circ$), independent of the indentation, the thickness
and the material of the sheet (for small indentations,
\cite{Cerda08032005}). The tip of the d-cone consists of a crescent-shaped
ridge where curvature and elastic stress focus. In numerical simulations it
was found that the  radius of the crescent $R_{\rm cres.}$ scales with the
thickness of the sheet $h$ and the radius of the container $R_{\rm cont.}$ as
$R_{\rm cres.}\sim h^{1/3}R_{\rm cont.}^{2/3}$. This dependence on the
container radius of the shape of the region near the tip  is not understood
\cite{RevModPhys.79.643}. As argued in this latter reference, it cannot be
explained by an analysis of the dominant contributions to the elastic energy,
which are: The bending energy from the region far away from the center, that
is well captured by modeling the d-cone as a developable surface there; and
the bending and stretching energy  part from a core region of size $O(h)$
where elastic strain is not negligible. The result of this (non-rigorous)
argument is an energy scaling $E\sim h^2 (C_1|\log h|+C_2)$.\\
Here, we discuss the scaling of the elastic energy with $h$ in a rigorous
setting.\\ 

The natural variational formulation is to minimize the
elastic energy of the sheet (which contains stretching and bending contributions)
in the class of deformations
$y:B_1\to\R^3$ which satisfy the obstacle constraint
\begin{equation*}
\mathrm{Im}(y)\cap \left\{x\in\R^3: x_1^2+x_2^2=R^2_{\rm cont.},x_3<H_{\rm
    cont.},y(0)=0\right\}=\emptyset
\end{equation*}
where $B_1=\{x\in\R^2:|x|\leq 1\}$ is the reference configuration of the
2-dimensional sheet, and  $H_{\rm cont.}$ is the height of the container. 
The derivation of a precise lower bound of the form 
$E\sim h^2 (C_1|\log h|+C_2)$ for this model looks very hard. 
Indeed even, for the much more severe constraint that
$y$ maps $B_1$ into a very small ball the only 
known rigorous lower bound is that
$\lim_{h \to \infty} E_h(y)/ h^2 =  \infty$ while it is conjectured
that the correct scaling is $E_h(y) \sim C h^{5/3}$ in this setting
(see \cite{MR2358334} for a discussion and a rigorous proof of the upper bound).
To make progress in rigorously  understanding the asymptotic  influence of the regularizing effect
of the thickness $h$ we follow recent work of 
Brandman, Kohn and Nguyen \cite{Brandman} and free ourselves from 
the specific obstacle type constraint and consider instead general Dirichlet problems
where the boundary conditions are given by an exact developable cone.

\medskip

This also partly motivated by   results that have been obtained in the physics literature for
a simpler model. In \cite{Cerda08032005}, the class of allowed deformations
$y$ is restricted to isometries with a singularity at the origin. These maps
are completely determined by their values on the boundary $\partial B_1$. In
order for these maps to be isometric away from the origin, the boundary values
have to be unit speed curves
\[
y|_{\partial B}=\gamma:\partial B_1\to S^2\,.
\]
This effectively reduces the problem from a 2-dimensional to a 1-dimensional
one. After a suitable ad-hoc renormalization of the bending energy (i.e., cutting out
a small ball around the origin where the energy density becomes singular), the elastic
energy is minimized as a functional of $\gamma$. In this simpler setting,
an obstacle of the above type is treatable. In fact, the shape of $\gamma$ in
the region where the sheet lifts off the rim of the cylinder is prescribed by
an ODE that can be solved more or less explicitly (for small
indentations, see \cite{Cerda08032005}).

\section{Setting and statement of the main theorem}
Let $B_r=\{x\in \R^2:| x|  <  r\}$, $A_r=B_r\setminus \overline{B_{r/2}}$ 
and for $y\in W^{2,2}(B_1,\R^3)$ let
\begin{align}
\label{elenergy}
E_h(y)= \int_{B_1} \left( |\nabla y^T\nabla y-Id|^2+h^2| \nabla^2
  y|^2\right)\d x\,.
\end{align}
Furthermore, for a curve $\gamma\in C^3(\partial B_1, \R^3)$ with
$|\gamma|=|\gamma'|=1$, let
\begin{equation}
V_{\gamma}=\left\{y\in W^{2,2}(B_1,\R^3):y|_{\partial B_1}=\gamma,\, y(0)=0\right\}\,.\non
\end{equation}
In  the following, consider such a $\gamma$ to be fixed. 
 By $\tilde
y(x)=|x|\gamma(\hat x)$, we denote the 1-homogeneous surface with boundary
values prescribed by $\gamma$.\\
\\
Existence of minimizers of $E_h$ in the class
$V_\gamma$ follows easily from the fact that
$E_h$ is coercive and convex in the highest derivatives
and the compact embedding
$W^{2,2} \hookrightarrow W^{1,4}$.
Our main result is
\begin{theorem}
\label{mainthm}
Suppose that $\gamma$ does not lie in a plane. 
Then for sufficiently small $h$ we have 
\begin{align*}
   C_1  \ln\frac{1}{h}-  C_1  \ln\left(\ln\frac{1}{h}\right)  -C_2
  \leq \frac{1}{h^2} \min_{y \in V_\gamma} E_h(y) \leq 
C_1\ln\frac{1}{h}+C_3, 
\label{thm2state}
\end{align*} 
where
\[
C_1=C_1(\gamma)=\frac{1}{\ln 2}\int_{B_1\setminus B_{1/2}} |\nabla^2\tilde
  y|^2\d x
\]
and $C_2,C_3$ only depend on $\gamma$. 
\end{theorem}

\bigskip
This improves work of Brandman, Kohn and Nguyen
\cite{Brandman} who showed that\\
$\liminf_{h \to 0} \frac{1}{h^2 \ln(1/h)} \min E_h \geq C_1/2$ and
$\limsup_{h \to 0} \frac{1}{h^2 \ln(1/h)} \min E_h \leq C_1$ and, very recently
in parallel to our work,
that  $\lim_{h \to 0} \frac{1}{h^2 \ln(1/h)} \min E_h = C_1$.
The result above shows that the deviation from the leading order logarithm
is at most a double logarithm. Indeed a natural conjecture is that this
error is or order 1, but we have not been able to prove or disprove this so far.

Our proof of the crucial  lower bound consists of three steps: First we  estimate 
the $L^\infty$-norm of $y$ in the ball $B_h$ (see Lemma
 \ref{lemma:suph} below).  Then we derive the
key estimate for the $L^2$-norm of $e=y-\tilde y$ on
dyadic rings $A_{2^{-j}}$ (see Lemma \ref{eL2} below). Finally  we argue that
since $y$ is close to $\tilde y$ in $L^2$, the bending energy 
of $y$ can be bounded from below by the bending energy of $\tilde y$, up to 
a small error. 

\section{Upper bound and $L^2$ estimate of $e=y-\tilde y$}

The proof of the upper bound is standard and we include it for the convenience of the reader.

\begin{lemma}
\label{upperbound}
\begin{align}
\inf_{y\in V_\gamma} E_h(y)<C_1 h^2\ln \frac{1}{h}+C_2h^2\non
\end{align}
where $C_1= \left(\ln 2\right)^{-1}\int_{B_{1}\setminus
  B_{1/2}}|\nabla^2 \tilde y|^2$.
\end{lemma}

\begin{proof}
Let $\phi: \R  \to \R$ be a $C^2$ function with $\phi(t) = t$ for $t \geq 1$ and
$\phi(t) = 0$ for $t \le 1/2$ and set
\begin{align*}
y_h(x) = h \phi\left(\frac{x}{h}\right) \gamma\left(\frac{x}{|x|}\right) .
\end{align*}
Then $y_h = \tilde y$ on $B_1 \setminus B_h$.
Hence $(\nabla y)^T \nabla y = Id$ in  $B_1 \setminus B_h$ and
\begin{equation}
\int_{B_1 \setminus B_h } \left( |\nabla y^T\nabla y-Id|^2+h^2| \nabla^2
  y|^2\right)\d x =  C_1 \ln \frac{1}{h} .
\end{equation}
Moreover in $B_h$ one has the estimates $|\nabla y_h| \leq C$ and $|\nabla^2 y_h|
 \leq C /h$. This implies the assertion. 

\end{proof}

Now we prove some auxiliary lemmas which  will
allow us to estimate $\sup_{B_h} |y|$.
This estimate will be needed in the proof of the
$L^2$ bound for  
of $y-\tilde y$  on dyadic rings.

\begin{lemma}
\label{supv}
Let $v\in W^{2,2}(B_h)$. Then
\[
\sup_{x\in B_h}\left|v(x)-v(0)-\left(\fint_{B_h(0)}  \nabla
    v(x')\d x'\right)\cdot x\right|\leq C h \| \nabla^2 v\|_{L^2(B_h)}\,.
\]
\end{lemma}

\begin{proof} 
For $h=1$ this follows from the embedding
$W^{2,2}(B_1)\hookrightarrow
C^0(B_1)$ and the Poincar\'e inequality.
For $h \neq 1$ the assertion follows by considering the rescaled
function $u(x) = \frac{1}{h} v(hx)$. 
\end{proof}

\begin{lemma}
\label{fintlem}
Let $w\in W^{1,2}(B_1)$, $0<\epsilon<1$. Then
\[
\left|\fint_{B_\epsilon}w\,\d x-\fint_{B_1}w\,\d x\right|\leq
C \left(\ln\frac{1}{\epsilon} \right)^{1/2}
\left(\int_{B_1}|\nabla w|^2\d x\right)^{1/2}\,.
\]
\end{lemma}
\begin{proof} 
The Poincar\'e inequality  for  $ w - \fint_{B_{R}}w$ implies that
for $ R \in [1/2, 1)$
\begin{equation} \label{eq:poincare1}
\left|\fint_{B_{R}} w \,  \d x-\fint_{B_{1}}  w \,\d x    \right|
\leq C \int_{B_1} | \nabla w  | \d x
\end{equation}
and scaling yields in particular
\begin{equation*}
\left|\fint_{B_{r/2}} w \,  \d x-\fint_{B_{r}}  w \,\d x    \right|
\leq C \frac{1}{r}  \int_{B_r} | \nabla w | \d x .
\end{equation*}
Apply this with $r = 2^{-k}$ for $k= 0, \ldots, n-1$ and
define 
\begin{equation}
f(x) := \sum_{k=0}^{n-1} 2^k \chi_{B_{2^{-k}}}.
\end{equation}
Then
\begin{equation*}
\left|\fint_{B_{2^{-n}}} w \,  \d x-\fint_{B_{1}}  w \,\d x    \right|
\leq C \int_{B_1}  f |\nabla w |  \d x  \leq \|f \|_{L^2(B_1)}   \|  \nabla w  \|_{L^2(B_1)} 
\end{equation*}
Now we have $ f \leq  2^{k+1}$ in $B_{2^{-k}} \setminus B_{2^{-k-1}}$ for $k \leq n-1$
and $ f \leq 2^{n+1}$ in $B_{2^{-n}}$. This implies that
$\|f\|_{L^2(B_1)} \leq C \sqrt{n}$. 
Choose $n$ such that $2^{-n}\geq \epsilon>2^{-(n+1)}$. Then scaling
of  \eqref{eq:poincare1} and the Cauchy-Schwarz inequality yield
\begin{equation*} 
\left|\fint_{B_{\epsilon}} w \,  \d x-\fint_{B_{2^{-n}}}  w \,\d x    \right|
\leq C 2^n \int_{B_{2^{-n} } }| \nabla w  | \d x 
\leq C \| \nabla w \|_{L^2(B_{2^{-n}})}
\end{equation*}
which completes the proof since $ n \leq  \ln_2 \frac{1}{\epsilon}$. 
\end{proof}

\begin{lemma}  \label{lemma:suph}
There exists constants such that for all $ 0 < h \leq 1/4$ and all $y \in V_\gamma$ we have
\begin{equation}
\sup_{B_h} |y| \leq C h + C h  \left( \ln \frac{1}{h} \right)^{1/2} \| \nabla^2 y\|_{L^2(B_1)}
\leq C h + C \left( \ln \frac{1}{h} \right)^{1/2} E_h^{1/2} (y) .
\end{equation}
\end{lemma}

\begin{proof}
This follows from Lemma \ref{supv} and Lemma \ref{fintlem} (applied with $w = \nabla v$) and
the following calculation 
\begin{equation*}
\int_{B_1} \nabla y  \, \d x= \int_{\partial B_1} y \otimes \nu \, \d S =  \int_{\partial B_1} \gamma \otimes \nu \, \d S
\end{equation*}
which yields 
\begin{equation*}
\left|  \fint_{B_1} \nabla y \, \d x \right| \leq \frac{2 \pi}{\pi} = 2 .
\end{equation*}
\end{proof}

We now come to the key estimate for the difference between a low energy map $y$ and $\tilde y$
in the $L^2$ norm on annuli. 
The idea is to look at fibres in radial direction in the domain,
i.e.~at line segments connecting the origin with $\partial B_1$. By the upper
bound on the elastic energy, the ``stretching'' of $y$ on such a line segment
is small. The boundary values of $y$ on the line segment are fixed as well, and so the deviation of  $y$  from
the straight line connecting the boundary points cannot be large.

\begin{lemma}
\label{eL2}
Let $h$ be small enough, $2h\leq r_0\leq 1$, and 
assume that 
$E_h(y) \leq 2 C_1 \ln \frac{1}{h}$.
Then
\begin{align*}
\int_{B_{r_0}\setminus B_{r_0/2}} |y-\tilde y|^2\d x\leq
C r_0^3h    \ln\frac{1}{h} 
+ C r_0^2  h^2  \left(  \ln \frac{1}{h} \right)^2 \, ,
\end{align*}
where $C$ is a constant that only depends on $\gamma$.
\end{lemma}

{\bf Remark} \quad Note that the second term on the right hand side of the estimate
is controlled by the first as long as $r_0 \geq h \ln \frac{1}{h}$. 

\begin{proof} 
We consider polar coordinates and set
\begin{equation}
\eta(r, \theta) = y(r \cos \theta, r \sin \theta), \qquad
e(r, \theta) = (y - \tilde y)( r\cos \theta, r \sin \theta)
\end{equation}
and we write $e' = (\partial/ \partial r) e$ and 
$\eta' = (\partial/ \partial r) \eta$.
By Fubini's theorem the maps $r \mapsto \eta(r, \theta)$
and $r \mapsto e(r, \theta)$ are weakly differentiable
for a.e. $\theta$. 

{\bf 1.} The key estimate is
\begin{equation}  \label{eq:keyestimate}
\int_0^{2 \pi} \int_h^1 |e'(\rho, \theta)|^2  \, \d \rho \, \d \theta \leq
C h \ln \frac{1}{h} \, .
\end{equation}
To prove this note that
$|\eta'|^2 = |\gamma + e'|^2 = 1 + 2 \gamma \cdot e' + |e'|^2$
which yields
\begin{equation*}
|e'|^2 = |\eta'|^2 - 1 - ( 2 \gamma \cdot e)'
\end{equation*}
since $\gamma$ depends only on $\theta$ but not on $r$. 
Using  the orthonormal basis $x/|x|$, $x^\perp/|x|$ we get the pointwise estimate
\begin{equation*}
|(\nabla y)^T \nabla y - Id|^2  \geq \left(  |\eta'|^2 - 1) \right)^2\,.
\end{equation*}
By the boundary condition we have $e(1, \theta) = 0$ and Lemma \ref{lemma:suph}
yields $|e(h, \theta)| \leq C h \ln \frac{1}{h}$ (since $|\tilde y(x)| = |x|$).
Thus an application of the Cauchy-Schwarz inequality gives
\begin{align*}
&\int_0^{2\pi} \int_h^1 |e'(\rho, \theta)|^2  \, \d \rho \, \d \theta \\
\leq & \int_0^{2\pi} \int_h^1 (|\eta'|^2 - 1)   \, \d \rho \, \d \theta  + C h \ln \frac{1}{h}\\
\leq &  \left( \int_0^{2\pi} \int_h^1 (|\eta'|^2 - 1)^2 \rho  \, \d \rho \, \d \theta \right)^{1/2}
 \left( \int_0^{2\pi} \int_h^1 \frac{1}{ \rho}  \, \d \rho \, \d \theta \right)^{1/2} + C h \ln \frac{1}{h}\\
 \leq & 
 E_h(y)^{1/2}   \left( 2 \pi  \ln \frac{1}{h} \right)^{1/2} +  C h \ln \frac{1}{h}
   \end{align*}
and \eqref{eq:keyestimate}
follows from the assumption on $E_h(y)$. 

{\bf 2. } The Cauchy-Schwarz inequality yields
\begin{equation*}
|e(r,\theta) - e(h, \theta)| \leq r^{1/2} \left( \int_h^r |e'(\rho, \theta)|^2 \, \d \rho  \right)^{1/2} \, .
\end{equation*}
for a.e. $\theta$. Taking the square, integrating over $\theta$ and using
\eqref{eq:keyestimate} we get
\begin{equation*}
\int_0^{2 \pi} |e(r, \theta) - e(h, \theta)|^2 \, \d \theta \leq 
C r  h \ln \frac{1}{h} \, .
\end{equation*}
By Lemma \ref{lemma:suph} we have 
\begin{equation*}
|e(h, \theta)|^2 \leq  C h^2 \left( \ln \frac{1}{h} \right)^2
\end{equation*}
and the assertion follows by integrating these two inequalities from $r_0/ 2$
to $r_0$ with respect to the measure $r \, \d r$.

\end{proof}

\section{Proof of Theorem \ref{mainthm}}
We first recall two interpolation inequalities for Sobolev functions.
\begin{lemma}
\label{boundarytobulk}
Let $\Omega\subset\R^n$ be a bounded Lipschitz domain, $v\in
W^{1,2}(\Omega)$, $u\in W^{2,2}(\Omega)$. Then
\begin{align}
\|v\|_{L^2(\partial \Omega)}^2\leq & C\|v\|_{L^2(\Omega)}\|v\|_{W^{1,2}(\Omega)} ,  \\
\| \nabla u\|_{L^2(\Omega)}^2\leq & 
C  \|u\|_{L^2(\Omega)}^2 +    C \|u\|_{L^2(\Omega)}    \|\nabla^2 u\|_{L^2(\Omega)}
 \label{eq:interpolation2}
\end{align}
where all constants $C$ only depend on $\Omega$.
\end{lemma}

\begin{proof}
The first inequality follows  from the continuity of the trace operator
$W^{1/2,2}(\Omega)\to L^2(\partial \Omega)$, and
the fact that
 $W^{1/2,2}(\Omega)$
 is a real interpolation space
of the  pair $(L^2(\Omega),W^{1,2}(\Omega))$ 
which yields
\begin{align*}
\|v\|_{L^2(\partial \Omega)}^2\leq  C\|v\|_{W^{1/2,2}(\Omega)}^2
\leq & C\|v\|_{L^2(\Omega)}\|v\|_{W^{1,2}(\Omega)}\,,
\end{align*}
see ~\cite{MR2424078}.
The second interpolation inequality follows directly from Theorem 5.2. 
in~\cite{MR2424078} which states that for all $\epsilon \le 1$
\begin{equation*}
\|\nabla u\|_{L^2(\Omega)}\leq K
\left(     \frac{1}{\epsilon} \|u\|_{L^2(\Omega)}+  \epsilon \|\nabla^2u \|_{L^2(\Omega)}\right)
\end{equation*}
If $ \|\nabla^2u \|_{L^2(\Omega)} \leq \|u\|_{L^2(\Omega)}$ one can take $\epsilon = 1$ to 
obtain \eqref{eq:interpolation2}, otherwise one takes
$\epsilon^2 =  \|u\|_{L^2(\Omega)}/  \|\nabla^2u \|_{L^2(\Omega)} $.

\end{proof}

\begin{proof}[Proof of Theorem \ref{mainthm}]
Let $M\in \N$ (to be chosen later). Recall that $e=y-\tilde y$.
We have
\begin{align}
\frac{1}{h^2}E_h(y)\geq&  \int_{B_1\setminus B_{2^{-M}}} |\nabla^2
y|^2\d x\non\\
=  & \int_{B_1\setminus B_{2^{-M}}} |\nabla^2\tilde
y|^2\d x-  2\int_{B_1\setminus B_{2^{-M}}}\, \nabla^2 e:\nabla^2\tilde
y\,\d x   
+ \int_{B_1\setminus B_{2^{-M}}} |\nabla^2 e
|^2\d x
 \non
\end{align}
Consider the second term on the right hand side
with the
integration restricted to the annulus $A_{r_0}$. We
define
$\hat e:A_1\rightarrow \R^3$ by $\hat e(x)=r_0^{-1}e(r_0x)$. Observe that
\begin{align}
\nabla^2 \hat e(x)= r_0(\nabla^2 e)(r_0x), \quad 
\nabla^2 \tilde y(x)= r_0(\nabla^2 \tilde y)(r_0x).\non
\end{align}
Since $\tilde y$ is $1$-homogeneous we get
\begin{align*}
 \int_{A_{1}}\, \nabla^2\hat e:\nabla^2\tilde
y\,\d x   
=
\int_{A_{1}}\, r_0^2(\nabla^2 e)(r_0x):(\nabla^2\tilde
y)(r_0x)\,\d x
=
\int_{A_{r_0}} \,\nabla^2 e:\nabla^2\tilde
y\,\d x  \, .
\end{align*}
We integrate by parts to obtain
\begin{align*}
\int_{A_{1}} \nabla^2 \hat e:\nabla^2\tilde
y\,\d x
= & 
\int_{A_{1}} \hat e_{,ij}\cdot\tilde y_{,ij}\,\d x
-\int_{A_{1}} \hat e_{,i}\cdot \tilde y_{,ijj}\d x+\int_{\partial
    A_{1}}\, \hat e_{,i}\cdot \tilde y_{,ij} \nu_j\,\d S
    \,,
\end{align*}
where $\nu$ is the unit outer normal on $\partial A_{1}$. Hence
\begin{align*}
  \left|\int_{A_{1}} \nabla^2 \hat e:\nabla^2\tilde
y\,\d x\right|\leq & \|\nabla^3\tilde y\|_{L^2(A_{1})}\|\nabla \hat e\|_{L^2(A_{1})}
+\|\nabla \hat e\|_{L^2(\partial A_{1})}\|\nabla^2\tilde y\|_{L^2(\partial
  A_{1})}\\
\leq & C\left(\|\nabla \hat e\|_{L^2(A_{1})}+\|\nabla \hat
  e\|_{L^2(A_{1})}^{1/2}\|\nabla \hat e\|_{W^{1,2}(A_{1})}^{1/2}\right)\\
 \leq & C\left(\|\nabla \hat e\|_{L^2(A_{1})}+\|\nabla \hat
  e\|_{L^2(A_{1})}^{1/2}\| \nabla^2 \hat e\|_{L^{2}(A_{1})}^{1/2}\right)
\end{align*}
where we used Lemma 
\ref{boundarytobulk}. 
Applying again the second inequality in that lemma, we get 
\begin{align}  \label{estimine}
  \left|\int_{A_{1}} \nabla^2 \hat e:\nabla^2\tilde
y\,\d x\right|
\leq 
C \left( 
\|\hat e\|_{L^2(A_{1})}+\|\hat e\|_{L^2(A_{1})}^{1/2} \|\nabla^2 \hat e\|_{L^{2}(A_{1})}^{1/2}+
\|\hat e\|_{L^2(A_{1})}^{1/4} \|\nabla^2 \hat e\|_{L^{2}(A_{1})}^{3/4}\right)  \, .
\end{align}

We apply Young's inequality $a b \leq \frac1p \delta^p  a^p + \frac{1}{p'} \delta^{-p'} b^{p'}$
with the pairs  $p = \frac43$, $p' = 4$ and $p = \frac85$,  $p' = \frac83$,  and $\delta^{-p'} = \frac{1}{4C}$
to obtain
\begin{align}  \label{estimine2}
  \left|\int_{A_{1}} \nabla^2 \hat e:\nabla^2\tilde
y\,\d x\right|
\leq 
C \left( \|\hat e\|_{L^2(A_{1})} + \|\hat e\|_{L^2(A_{1})}^{2/3} + \|\hat e\|_{L^2(A_{1})}^{2/5})  \right)+ 
\frac12  \| \nabla^2 \hat e\|^2_{L^2(A_1)}  \, .
\end{align}
Now we undo the rescaling.    We have 
\begin{align*}
\|\hat e\|_{L^2(A_1)}= r_0^{-2}\|e\|_{L^2(A_{r_0})}, \quad
\|\nabla^2 \hat e\|_{L^2(A_1)}= \|\nabla^2 e\|_{L^2(A_{r_0})}\,.
\end{align*}
By Lemma  \ref{eL2}
\begin{equation}
r_0^{-2} \|e\|_{L^{2}(A_{r_0})}\leq C
\left( \frac {h}{r_0} \right)^{1/2} \left(\ln\frac{1}{h}\right)^{1/2}
\end{equation}
as long as 
\begin{equation}  \label{eq:lowerbound_r0}
r_0 \ge  h \ln \frac{1}{h}  \, .
\end{equation}
Thus
\begin{align}  \label{estimine3}
  \left|\int_{A_{r_0}} \nabla^2  e:\nabla^2\tilde
y\,\d x\right|
\leq 
C \left( \frac {h}{r_0} \right)^{1/5} \left(\ln\frac{1}{h}\right)^{1/5} + \frac12 \| \nabla^2 e \|_{L^2(A_{r_0})}^2
\end{align}
as long as \eqref{eq:lowerbound_r0} holds.

Choose $M$ such that
\begin{equation}  \label{eq:choice_M}
\log_2 \frac{1}{h}-\log_2\ln\frac{1}{h}\leq M\leq \log_2
\frac{1}{h}-\log_2\ln\frac{1}{h}+1\,.
\end{equation}
This implies that
\begin{equation*}
2^{-(M-1)} \geq  h \ln \frac{1}{h} \, .
\end{equation*}
In particular for sufficiently small $h$ and for $r_0 \geq  2^{-(M-1)}$
the inequality \eqref{eq:lowerbound_r0} holds and we  also have $r_0 \geq 2h$. 
We thus get
\begin{align}  \label{eq:final_estimate}
\int_{B_1\setminus B_{2^{-M}}} & |\nabla^2  y|^2 \, \d x   \nonumber \\
\geq & 
\int_{B_1\setminus B_{2^{-M}}}  |\nabla^2 \tilde y
|^2\d x- 2 \int_{B_1\setminus B_{2^{-M}}} \nabla^2\tilde y: \nabla^2
  e\,\d x    +  \int_{B_1\setminus B_{2^{-M}}}  |\nabla^2  e|^2 \, \d x
\nonumber \\
\geq & 
 M\left(\int_{A_1}|\nabla^2 \tilde y|^2\d x \right) - 2 C \sum_{j=0}^{M-1}
 2^{j/5}h^{1/5}\left(\ln\frac{1}{h}\right)^{1/5}  \nonumber \\
\geq & M\left(\int_{A_1}|\nabla^2 \tilde y|^2\d x \right)- 20 C \, 2^{(M-1)/5}h^{1/5}\left(\ln\frac{1}{h}\right)^{1/5}
\nonumber  \\
\geq & \frac{1}{\ln
  2}\left(\ln\frac{1}{h}- \ln\left(\ln\frac{1}{h}\right)\right)
\left(\int_{A_1} |\nabla^2 \tilde y|^2  \, \d x\right)-  20  C\,
\end{align}
This completes the proof of
Theorem \ref{mainthm}.

\end{proof}

{\bf Remark.}
 The estimate \eqref{eq:final_estimate}  in connection with the upper bound
on $E_h$ shows that 
for any $y$ with $E_h(y) \leq C_1 \ln \frac{1}{h} + C_3$  (and in particular for any minimizer
of $E_h$) we have
\begin{align}
 \int_{B_1} |\nabla y^T\nabla y-Id|^2 \, \d x   &\le  C_1 \ln \left( \ln \frac{1}{h} \right) + C_4, \\
 \int_{B_{2^{-M}}} |\nabla^2  y|^2 \, \d x  &\le  C_1 \ln \left( \ln \frac{1}{h} \right) + C_4,
 \end{align}
where $M$ is as in \eqref{eq:choice_M} and \eqref{eq:final_estimate}.
In addition we may assume that  \eqref{estimine3} holds with the term 
$\frac14 \| \nabla^2 e \|_{L^2(A_{r_0})}^2$ instead of 
$\frac12 \| \nabla^2 e \|_{L^2(A_{r_0})}^2$ on the right hand side. 
Then the same argument as in \eqref{eq:final_estimate} yields in addition that
\begin{equation}
\frac12 \int_{B_1\setminus B_{2^{-M}}}  |\nabla^2  e|^2 \, \d x \le    C_1 \ln \left( \ln \frac{1}{h} \right) + C_4
\end{equation}
These additional estimates can be used to improve the prefactor 
of $ \ln \left( \ln \frac{1}{h} \right) $ in the lower bound from $1$ to $\frac12 + \varepsilon$.

\bibliographystyle{plain}
\bibliography{dcone}

\end{document}